\documentclass{amsart}
\usepackage{amsmath,amsthm,amssymb}
\usepackage[T1]{fontenc}
\newtheorem{theorem}{Theorem}
\newtheorem{proposition}[theorem]{Proposition}
\newtheorem{lemma}[theorem]{Lemma}
\newtheorem{corollary}[theorem]{Corollary}

\theoremstyle{remark}
\newtheorem{remark}{Remark}
\newtheorem*{definition}{Definition}

\newcommand{\vep}{\varepsilon}
\newcommand{\R}{{\mathbb{R}}}

\newcommand{\Q}{{\mathbb{Q}}}
\newcommand{\C}{{\mathbb{C}}}
\newcommand{\Z}{{\mathbb{Z}}}
\newcommand{\N}{{\mathbb{N}}}

\newcommand{\Arg}{\operatorname{Arg}}

\begin{document}
\title[Density of mild mixing property for  Abelian
differentials]{Density of mild mixing property for vertical flows
of Abelian differentials}
\author[K. Fr\k{a}czek]{Krzysztof Fr\k{a}czek}
\address{Faculty of Mathematics and Computer Science, Nicolaus
Copernicus University, ul. Chopina 12/18, 87-100 Toru\'n, Poland}
\address{Institute of Mathematics\\
Polish Academy of Science\\
ul. \'Sniadeckich 8\\
00-956 Warszawa, Poland } \email{fraczek@mat.uni.torun.pl}
\date{}


\thanks{Research partially supported by MNiSzW grant N N201
384834 and Marie Curie "Transfer of Knowledge" program, project
MTKD-CT-2005-030042 (TODEQ)}

\begin{abstract}
We prove that if $g\geq 2$ then the set of all Abelian
differentials $(M,\omega)$ for which the vertical flow is mildly
mixing is dense in every stratum of the moduli space
$\mathcal{H}_g$. The proof is based on a sufficient condition in
\cite{Fr-Le-Le}  for special flows over irrational rotations and
under piecewise constant roof functions to be mildly mixing.
\end{abstract}
\maketitle

\section{Abelian differentials and direction flows}
For every natural $g\geq 2$ let $\mathcal{H}_g$ stand for the
moduli space of equivalence classes of pairs $(M,\omega)$ where
$M$ is a compact Riemann surface of genus $g$ and $\omega$ is a
nonzero holomorphic 1-form on $M$ (an Abelian differential). Two
pairs $(M,\omega)$ and $(M',\omega')$ are identified if they are
mapped to one another by a conformal homeomorphism. The space
$\mathcal{H}_{g}$ is naturally stratified by the subsets
$\mathcal{H}_g(m_1,\ldots,m_{\kappa})$ of Abelian differentials
whose zeros have multiplicities $m_1,\ldots,m_{\kappa}$. By the
Euler-Poincar\'e formula $m_1+\ldots+m_{\kappa}=2g-2$. Every
stratum $\mathcal{H}_g(m_1,\ldots,m_{\kappa})$ is  a
complex-analytic orbifold of dimension $2g+\kappa-1$. Moreover,
$\mathcal{H}_g(m_1,\ldots,m_{\kappa})$ possesses a natural
Lebesgue measure $\nu$. Let us denote by
$(\mathcal{U}_s)_{s\in\R}$ the periodic continuous flow on
$\mathcal{H}_{g}$ defined by $\mathcal{U}_s(\omega)=e^{is}\omega$.

For every $\theta\in\C$ such that $|\theta|=1$, the Abelian
differential $\omega$ determines the direction field
$v_\theta:M\to TM$ so that $\omega(v_\theta)=\theta$ for all
points of $M$ except the zeros of $\omega$ which are singular for
$v_{\theta}$. By the direction flow we will mean the flow
$\mathcal{F}^{\theta}=\mathcal{F}^{\omega,\theta}$ generated by
$v_\theta$. The flows $\mathcal{F}^{1}$ and $\mathcal{F}^{i}$ are
called horizontal and vertical respectively. Direction flows
preserve the volume form
$\frac{i}{2}\omega\wedge\overline{\omega}$ on $M$ which vanishes
only at zeros of $\omega$. This form determines a finite volume
measure $\mu_{\omega}$ which is invariant for all direction flows.

 A
separatrix of $\mathcal{F}^{\theta}$ joining two singularities
(not necessarily distinct) is called a saddle connection of
$\mathcal{F}^{\theta}$. Recall that in every stratum for a.e.\
Abelian differential $(M,\omega)$ the vertical and the horizontal
flows have no saddle connections.

We are interested in ergodic (mixing) properties of the vertical
flow $\mathcal{F}^i$ for $g\geq 2$. Avila and Forni proved in
\cite{Av-Fo} that for $\nu$-almost all
$(M,\omega)\in\mathcal{H}_g(m_1,\ldots,m_{\kappa})$ the vertical
flow is weakly mixing with respect to the measure $\mu_{\omega}$.
It follows from Katok's result in \cite{Ka} that direction flows
are never strongly mixing.

In this paper we will restrict our attention to the mild mixing
property for $\mathcal{F}^i$. A finite measure--preserving
dynamical system is {\em mildly mixing} (see \cite{Fu-We}) if its
Cartesian product with an arbitrary ergodic conservative (finite
or infinite) measure-preserving dynamical system remains ergodic.
It is an immediate observation that the strong mixing of a
dynamical system implies its mild mixing and mild mixing implies
weak mixing. Recall that a measure-preserving flow
$(T_t)_{t\in\R}$ on $(X,\mathcal{B},\mu)$ is {\em rigid} if there
exists $t_n\to+\infty$ such that $\mu(T_{t_n}^{-1}A\triangle A)\to
0$ for all $A\in\mathcal{B}$. It was proved in \cite{Fu-We} that a
finite measure--preserving flow is mildly mixing if and only if it
has no non-trivial rigid factors. Using the same methods as in the
proof of Theorem~1.3 in \cite{Ve2}, one can prove that for almost
every $(M,\omega)\in\mathcal{H}_g(m_1,\ldots,m_{\kappa})$ the
vertical flow is rigid. It follows that the set $\mathcal{H}_{mm}$
of $(M,\omega)\in\mathcal{H}_g(m_1,\ldots,m_{\kappa})$ for which
the vertical flow is mildly mixing is of measure zero.
Nevertheless, we prove that  $\mathcal{H}_{mm}$ is dense in every
stratum $\mathcal{H}_g(m_1,\ldots,m_{\kappa})$ (see
Theorem~\ref{milddensegen}).

The proof of the  density of $\mathcal{H}_{mm}$ is based on three
components: a polygonal representation of Abelian differentials
 described in Section~\ref{abdif} where we follow
\cite{Vi0}, the Rauzy-Veech induction (Section~\ref{RVinduction})
and a sufficient condition  in \cite{Fr-Le-Le}  for special flows
built over irrational rotations and under piecewise constant roof
functions to be mildly mixing (see Proposition~\ref{mild}). The
proof consists of two main steps. In the first step, using the
Rauzy-Veech induction, we prove that a typical Abelian
differential is approximated by Abelian differentials whose
vertical flows are isomorphic to step special flows built over
three intervals exchange transformations and under roof functions
constant on the exchanged intervals (see Lemma~\ref{prostogen}).
In the second step we apply the main result of \cite{Fr-Le-Le}. It
says that a special flow built over an irrational circle rotation
by $\alpha$ and under a three steps roof function (with one jump
at $1-\alpha$ and one jump at some point $\xi$) is mildly mixing
for a dense set of the data ($\alpha$, $\xi$ and heights of the
steps). Using the Rauzy-Veech induction again, it follows that the
same result holds for step special flows over exchanges of three
intervals, i.e.\ such special flows are mildly mixing for a dense
set of data (see Corollary~\ref{wniosek}).

\section{Interval exchange transformations and a construction of Abelian
differentials}\label{abdif} In this section we briefly describe a
standard construction of Abelian differentials. For more details
we refer the reader to \cite{Vi0} and \cite{Vi}.

\subsection{Interval exchange transformations} Let $\mathcal{A}$ be a $d$-element alphabet and let
$\pi=(\pi_0,\pi_1)$ be a pair of bijections
$\pi_\vep:\mathcal{A}\to\{1,\ldots,d\}$ for $\vep=0,1$. We adopt
the notation from \cite{Vi0}. The set of all such pairs we will
denote by $\mathcal{P}_{\mathcal{A}}$. Denote by
$\mathcal{P}^0_{\mathcal{A}}$ the subset of irreducible pairs,
i.e.\ such that
$\pi_1\circ\pi_0^{-1}\{1,\ldots,k\}\neq\{1,\ldots,k\}$ for $1\leq
k<d$. Let $\mathcal{P}^*_{\mathcal{A}}$ stand for the set of
irreducible pairs such that
$\pi_1\circ\pi_0^{-1}(k+1)\neq\pi_1\circ\pi_0^{-1}(k)+1$ for
$1\leq k<d$.

 Let us consider
$\lambda=(\lambda_\alpha)_{\alpha\in\mathcal{A}}\in
\R_+^{\mathcal{A}}\setminus\{0\}$, where $\R_+=[0,+\infty)$. Let
\[|\lambda|=\sum_{\alpha\in\mathcal{A}}\lambda_\alpha,\;\;\;
I=\left[0,|\lambda|\right)\text{ and
}I_{\alpha}=\left[\sum_{\pi_0(\beta)<\pi_0(\alpha)}\lambda_\beta,\sum_{\pi_0(\beta)\leq\pi_0(\alpha)}\lambda_\beta\right).\]
Then $|I_\alpha|=\lambda_\alpha$. Let $\Omega_\pi$ stand the
matrix $[\Omega_{\alpha\,\beta}]_{\alpha,\beta\in\mathcal{A}}$
given by
\[\Omega_{\alpha\,\beta}=
\left\{\begin{array}{cl} +1 & \text{ if
}\pi_1(\alpha)>\pi_1(\beta)\text{ and
}\pi_0(\alpha)<\pi_0(\beta)\\
-1 & \text{ if }\pi_1(\alpha)<\pi_1(\beta)\text{ and
}\pi_0(\alpha)>\pi_0(\beta)\\
0& \text{ in all other cases.}
\end{array}\right.\]
Given $(\lambda,\pi)\in\R_+^\mathcal{A}\times
\mathcal{P}^0_{\mathcal{A}}$ let
$T_{(\lambda,\pi)}:[0,|\lambda|)\rightarrow[0,|\lambda|)$
 stand for  the {\em interval
exchange transformation} (IET) on $d$ intervals $I_\alpha$,
$\alpha\in\mathcal{A}$, which are rearranged according to the
permutation $\pi$, i.e.\ $T_{(\pi,\lambda)}x=x+w_\alpha$ for $x\in
I_\alpha$, where $w=\Omega_\pi\lambda$.

\begin{definition}Let $\partial I_{\alpha}$ stand for the left end point of the
interval $I_\alpha$. A pair ${(\lambda,\pi)}$ satisfies the {\em
Keane condition} if $T_{(\lambda,\pi)}^m\partial I_{\alpha}\neq
\partial I_{\beta}$ for all $m\geq 1$ and for all
$\alpha,\beta\in\mathcal{A}$ with $\pi_0(\beta)\neq 1$.
\end{definition}
It was proved by Keane in \cite{Ke} that if
$\pi\in\mathcal{P}^0_{\mathcal{A}}$ then for almost every
$\lambda$ the pair ${(\lambda,\pi)}$ satisfies the  Keane
condition.

\subsection{Construction of Abelian differentials} For each $\pi\in\mathcal{P}^0_{\mathcal{A}}$
denote by $\mathcal{T}^+_\pi$ the set of vectors
$\tau=(\tau_\alpha)_{\alpha\in\mathcal{A}}\in \R^{\mathcal{A}}$
such that
\begin{equation}\label{dodat}
\sum_{\pi_0(\alpha)\leq k}\tau_{\alpha}>0\text{ and
}\sum_{\pi_1(\alpha)\leq k}\tau_{\alpha}<0\text{ for all }1\leq
k<d.
\end{equation}
Denote by $\mathcal{T}^+_{\pi,\lambda}$ the set of $\tau\in
\mathcal{T}^+_\pi$ for which
\begin{equation}\label{dodatla}
\lambda_{\pi^{-1}_\vep(k)}=\lambda_{\pi^{-1}_\vep(k+1)}=0\Longrightarrow
\tau_{\pi^{-1}_\vep(k)}\cdot\tau_{\pi^{-1}_\vep(k+1)}>0\text{ for
}1\leq k<d,\;\vep=0,1.
\end{equation}
Of course, $\mathcal{T}^+_\pi$ and $\mathcal{T}^+_{\pi,\lambda}$
are open convex cones.

Assume that $\tau\in \mathcal{T}^+_{\pi,\lambda}$ and set
$\zeta_{\alpha}=\lambda_\alpha+i\tau_\alpha\in\C$ for each
$\alpha\in\mathcal{A}$. Let $\Gamma(\pi,\lambda,\tau)$ stand for
the closed curve on $\C$ formed by concatenation of vectors
\[\zeta_{\pi_0^{-1}(1)},\zeta_{\pi_0^{-1}(2)},\ldots,\zeta_{\pi_0^{-1}(d)},
-\zeta_{\pi_1^{-1}(d)},-\zeta_{\pi_1^{-1}(d-1)},\ldots,-\zeta_{\pi_1^{-1}(1)}\]
with starting point at zero.  The curve $\Gamma(\pi,\lambda,\tau)$
determines a polygon $P(\pi,\lambda,\tau)$ on $\C$ with $2d$ sides
which has $d$ pairs of parallel sides with the same length.
Condition (\ref{dodat}) means that the first $d-1$ vertices of the
polygon $\sum_{k=1}^j\zeta_{\pi_0^{-1}(k)}$, $j=1,\ldots,d-1$ are
on the upper half-plane and the last $d-1$ vertices
$\sum_{k=1}^j\zeta_{\pi_1^{-1}(k)}$, $j=1,\ldots,d-1$ are on the
lower half-plane.
\begin{definition}(see \cite{Vi0} and \cite{Zo})
The {\em suspension surface} $M(\pi,\lambda,\tau)$ is a compact
surface obtained by the identification of the sides of the polygon
$P(\pi,\lambda,\tau)$ in each pair of parallel sides. The surface
$M(\pi,\lambda,\tau)$ possesses a natural complex structure
inherited from $\C$ and a holomorphic $1$-form $\omega$ determined
by the form $dz$. Therefore $M(\pi,\lambda,\tau)$ can be treated
as an element of a moduli space.
\end{definition}
The zeros of $\omega$ correspond to the vertices of the polygon
$P(\pi,\lambda,\tau)$  and the vertical flow $\mathcal{F}^i$ moves
up each point of $P(\pi,\lambda,\tau)$ vertically at the unit
speed. Note that for every $s\in\R$, taking
$\lambda_s+i\tau_s=e^{is}(\lambda+i\tau)$,
\begin{equation}\label{obrot}
\text{ if }\lambda_s\in \R_+^{\mathcal{A}}\text{ and
}\tau_s\in\mathcal{T}^+_{\pi}\text{ then
}M(\pi,\lambda_s,\tau_s)=\mathcal{U}_sM(\pi,\lambda,\tau).
\end{equation}

\subsection{Zippered rectangles and a special representation of the vertical flow}
Suspension surfaces can be defined in the terms of zippered
rectangles introduced by Veech \cite{Ve1}. For every
$(\pi,\lambda,\tau)$ with $\tau\in \mathcal{T}^+_{\pi,\lambda}$
let us consider the vector $h=h(\tau)=-\Omega_\pi\tau$. In view of
(\ref{dodat}), $h\in\R_+^{\mathcal{A}}$. Here the surface
$M(\pi,\lambda,\tau)$ is obtained from the rectangles
$I_\alpha\times[0,h_\alpha]$, $\alpha\in\mathcal{A}$ by an
appropriate identification of parts of their sides. For example,
the interval $I_\alpha\times\{h_\alpha\}$ is identified by a
translation with $T_{(\pi,\lambda)}I_\alpha\times\{0\}$  for all
$\alpha\in\mathcal{A}$ (see \cite{Ve1} for details).

In this representation the vertical flow $\mathcal{F}^i$ moves up
each point of zippered rectangles vertically at the unit speed
which yields the following fact.
\begin{lemma}\label{spec}If $\tau\in \mathcal{T}^+_{\pi,\lambda}$ then the vertical flow on
$M(\pi,\lambda,\tau)$ has a special representation over the
interval exchange transformation $T_{(\pi,\lambda)}$ and under the
roof function
$$f_h:I\to\R_+,\;\;
f_h=\sum_{\alpha\in\mathcal{A}}h_{\alpha}\chi_{I_{\alpha}},$$
i.e.\ the vertical flow and the special flow
$T_{(\pi,\lambda)}^{f_h}$ are isomorphic as measure-preserving
systems.
\end{lemma}
We will also need the following results.
\begin{proposition}[see Proposition~3.30 in
\cite{Vi} or \cite{Ve3}]\label{realizacja} If $m_i>0$ for
$i=1,\ldots,\kappa$ then $\nu$-almost every
$(M,\omega)\in\mathcal{H}_g(m_1,\ldots,m_{\kappa})$ may be
represented in the form $M(\pi,\lambda,\tau)$, where
$\#\mathcal{A}=2g+\kappa-1$.
\end{proposition}
\begin{remark}\label{gestpow}
By the proof of Proposition~3.30 in \cite{Vi}, we can choose $\pi$
from  $\mathcal{P}^*_{\mathcal{A}}$.
\end{remark}
\begin{proposition}[see~\cite{Ve1} and \cite{Ve3}] \label{contm}
For fixed $\pi$ all Abelian differentials
$M(\pi,\lambda,\tau)$ lie in the same stratum
$\mathcal{H}_g(m_1,\ldots,m_{\kappa})$  and the map
\[\hat{\mathcal{H}}(\pi)=\{\pi\}\times\left(\R_+\setminus\{0\}\right)^{\mathcal{A}}\times \mathcal{T}^+_{\pi}\ni
(\pi,\lambda,\tau)\mapsto
M(\pi,\lambda,\tau)\in\mathcal{H}_g(m_1,\ldots,m_{\kappa})\] is
continuous.
\end{proposition}

\section{Rauzy-Veech induction}\label{RVinduction}
In this section we describe the Rauzy-Veech induction
renormalization procedure introduced for IETs by Rauzy in
\cite{Ra} and extended to zippered rectangles by Veech in
\cite{Ve1}.

Let
$(\pi,\lambda)\in\mathcal{P}^0_{\mathcal{A}}\times(\R_+^{\mathcal{A}}\setminus\{0\})$
be a pair such that
$\lambda_{\pi_0^{-1}(d)}\neq\lambda_{\pi_1^{-1}(d)}$. Set
\[\vep(\lambda,\pi)=\left\{
\begin{array}{ccl}
0&\text{ if }&\lambda_{\pi_0^{-1}(d)}>\lambda_{\pi_1^{-1}(d)}\\
1&\text{ if }&\lambda_{\pi_0^{-1}(d)}<\lambda_{\pi_1^{-1}(d)}.
\end{array}
\right.\] We say that $(\pi,\lambda)$ has type
$\vep(\lambda,\pi)$. For $\vep=0,1$ let
$R_\vep:\mathcal{P}^0_{\mathcal{A}}\to
\mathcal{P}^0_{\mathcal{A}}$ be defined by
$R_\vep(\pi_0,\pi_1)=(\pi'_0,\pi'_1)$, where
\begin{eqnarray*}\pi'_\vep(\alpha)&=&\pi_\vep(\alpha)\text{ for all }\alpha\in\mathcal{A}\text{ and }\\
\pi'_{1-\vep}(\alpha)&=&\left\{
\begin{array}{cll}
\pi_{1-\vep}(\alpha)& \text{ if
}&\pi_{1-\vep}(\alpha)\leq\pi_{1-\vep}\circ\pi^{-1}_\vep(d)\\
\pi_{1-\vep}(\alpha)+1& \text{ if
}&\pi_{1-\vep}\circ\pi^{-1}_\vep(d)<\pi_{1-\vep}(\alpha)<d\\
\pi_{1-\vep}\pi^{-1}_\vep(d)+1& \text{ if
}&\pi_{1-\vep}(\alpha)=d.\end{array} \right.
\end{eqnarray*}
Moreover, let
$\Theta_{\pi,\vep}=[\Theta_{\alpha\,\beta}]_{\alpha,\beta\in\mathcal{A}}$
stand for the matrix
\[\Theta_{\alpha\,\beta}=\left\{
\begin{array}{ll}
1& \text{ if
}\alpha=\beta\\
1& \text{ if
}\alpha=\pi_{1-\vep}^{-1}(d)\text{ and }\beta=\pi_{\vep}^{-1}(d)\\
0& \text{ in all other cases.}\end{array} \right.\] The {\em
Rauzy-Veech induction} of $T_{(\lambda,\pi)}$ is the first return
map $T'$ of $T_{(\lambda,\pi)}$ to the interval
$$\left[0,|\lambda|-\min(\lambda_{\pi^{-1}_0(d)},\lambda_{\pi^{-1}_1(d)})\right).$$
As it was shown by Rauzy in \cite{Ra}, $T'$ is also an IET on
$d$-intervals, hence $T'=T_{(\lambda',\pi')}$ for some
$(\lambda',\pi')\in\mathcal{P}^0_{\mathcal{A}}\times(\R_+^{\mathcal{A}}\setminus\{0\})$.
Moreover,
\[(\lambda',\pi')=(R_{\vep}\pi,\Theta^{-1*}_{\pi,\vep}\lambda),\text{ where }\vep=\vep(\pi,\lambda),\]
and $B^*$ denotes the conjugate transpose of $B$. This
renormalization procedure determines the transformation
$$\hat{R}:\mathcal{P}^0_{\mathcal{A}}\times\left(\R_+^{\mathcal{A}}\setminus\{0\}\right)\to
\mathcal{P}^0_{\mathcal{A}}\times\left(\R_+^{\mathcal{A}}\setminus\{0\}\right),\;\;
\hat{R}(\pi,\lambda)=(R_{\vep(\pi,\lambda)}\pi,\Theta^{-1*}_{\pi,\vep(\pi,\lambda)}\lambda)$$
whenever $\lambda_{\pi_0^{-1}(d)}\neq\lambda_{\pi_1^{-1}(d)}$.
Therefore the map  $\hat{R}$ is well defined for all
$(\pi,\lambda)$ satisfying the Keane condition. Moreover,
$\hat{R}(\pi,\lambda)$ fulfills the Keane condition for each such
$(\pi,\lambda)$. Consequently, $\hat{R}^n(\pi,\lambda)$ is well
defined for all $n\geq 1$ and for all $(\pi,\lambda)$ satisfying
the Keane condition (see \cite{Yo} for details).

\subsection{Rauzy graphs and Rauzy-Veech cocycle}
\begin{definition}Let us consider the relation $\sim$ on
$\mathcal{P}^0_{\mathcal{A}}$ for which $\pi\sim\pi'$ if there
exists $(\vep_1,\ldots,\vep_k)\in\{0,1\}^*$ such that
$\pi'=R_{\vep_k}\circ\ldots\circ R_{\vep_1}\pi$. Then $\sim$ is an
equivalence relation; its equivalence classes are called  {\em
Rauzy classes}.
\end{definition}
Of course, for each Rauzy class $C\subset
\mathcal{P}^0_{\mathcal{A}}$, the set $C\times \R_+^{\mathcal{A}}$
is $\hat{R}$--invariant.
\begin{definition}
A pair $\pi\in\mathcal{P}^0_{\mathcal{A}}$ is called {\em
standard} if  ${\pi}_1\circ{\pi}_0^{-1}(1)=d$ and
${\pi}_1\circ{\pi}_0^{-1}(d)=1$.
\end{definition}

\begin{proposition}[see \cite{Ra}]\label{Rauzy}
Every Rauzy class contains a standard pair.
\end{proposition}
 Denote by $\Theta:C\times
\R_+^{\mathcal{A}}\to GL(d,\Z)$ the {\em Rauzy-Veech cocycle}
$$\Theta(\pi,\lambda)=\Theta_{\pi,\vep(\pi,\lambda)}.$$
If $(\pi',\lambda')=\hat{{R}}^n(\pi,\lambda)$ then
$\lambda'=\Theta^{(n)}(\pi,\lambda)^{-1*}\lambda$, where
\[\Theta^{(n)}(\pi,\lambda)=\Theta(\hat{R}^{n-1}(\pi,\lambda))\cdot\Theta(\hat{R}^{n-2}(\pi,\lambda))\cdot\ldots\cdot\Theta(\hat{R}(\pi,\lambda))\cdot\Theta(\pi,\lambda)\]
\begin{remark}\label{powrot}
For every $\lambda\in(\R_+\setminus\{0\})^{\mathcal{A}}$ we have
$\vep(\pi,\Theta_{\pi,\vep}^*\lambda)=\vep$. Indeed,
\begin{eqnarray*}\left(\Theta_{\pi,\vep}^*\lambda\right)_{\pi^{-1}_{\vep}(d)}&=&\sum_{\alpha\in\mathcal{A}}(\Theta_{\pi,\vep})_{\alpha\,\pi^{-1}_{\vep}(d)}\lambda_{\alpha}=
\lambda_{\pi^{-1}_{\vep}(d)}+\lambda_{\pi^{-1}_{1-\vep}(d)},\\
\left(\Theta_{\pi,\vep}^*\lambda\right)_{\pi^{-1}_{1-\vep}(d)}&=&\sum_{\alpha\in\mathcal{A}}(\Theta_{\pi,\vep})_{\alpha\,\pi^{-1}_{1-\vep}(d)}\lambda_{\alpha}=
\lambda_{\pi^{-1}_{1-\vep}(d)},
\end{eqnarray*} and hence
$\left(\Theta_{\pi,\vep}^*\lambda\right)_{\pi^{-1}_{\vep}(d)}>\left(\Theta_{\pi,\vep}^*\lambda\right)_{\pi^{-1}_{1-\vep}(d)}$.
Therefore
$\hat{R}(\pi,\Theta_{\pi,\vep}^*\lambda)=(R_\vep\pi,\lambda)$.

Now assume that $(\pi',\lambda')=\hat{R}(\pi,\lambda)$. Then for
every $\lambda''\in(\R_+\setminus\{0\})^{\mathcal{A}}$,
\[\hat{R}(\pi,\Theta(\pi,\lambda)^*\lambda'')=\hat{R}(\pi,\Theta_{\pi,\vep(\pi,\lambda)}^*\lambda'')=(R_{\vep(\pi,\lambda)},\lambda'')=(\pi',\lambda'').\]
It follows that for every $(\pi,\lambda)$, $n\geq 1$ and
$\lambda''\in(\R_+\setminus\{0\})^{\mathcal{A}}$
$$(\pi',\lambda')=\hat{R}^n(\pi,\lambda)\text{ implies }
\hat{R}^n(\pi,\Theta^{(n)}(\pi,\lambda)^*\lambda'')=(\pi',\lambda'').$$
\end{remark}

\subsection{Extended Rauzy-Veech induction}
For every Rauzy class $C\subset \mathcal{P}^0_{\mathcal{A}}$ let
$$\hat{\mathcal{H}}(C)=\{(\pi,\lambda,\tau):\pi\in
C,\lambda\in\R_+^{\mathcal{A}}\setminus\{0\},\tau\in
\mathcal{T}^+_\pi\}.$$ By the {\em extended Rauzy-Veech induction}
we mean the map
$\hat{\mathcal{R}}:\hat{\mathcal{H}}(C)\to\hat{\mathcal{H}}(C)$,
$$\hat{\mathcal{R}}(\pi,\lambda,\tau)=(R_{\vep(\pi,\lambda)}\pi,
\Theta^{-1*}_{\pi,\vep(\pi,\lambda)}\lambda,\Theta^{-1*}_{\pi,\vep(\pi,\lambda)}\tau)=(\hat{R}(\pi,\lambda),\Theta^{-1*}(\pi,\lambda)\tau).$$
By Lemma~18.1 in \cite{Vi0}, if
$(\pi',\lambda')=\hat{R}(\pi,\lambda)$ then
$\Theta^{-1*}(\pi,\lambda)\tau\in \mathcal{T}^+_{\pi'}$, and hence
$\hat{\mathcal{R}}:\hat{\mathcal{H}}(C)\to\hat{\mathcal{H}}(C)$ is
well defined almost everywhere. Moreover,  for every $n\geq 1$
\[\text{
if
}(\pi',\lambda',\tau')=\hat{\mathcal{R}}^n(\pi,\lambda,\tau)\text{
then }\lambda'=\Theta^{(n)}(\pi,\lambda)^{-1*}\lambda\text{ and
}\tau'=\Theta^{(n)}(\pi,\lambda)^{-1*}\tau.\]

\begin{lemma}[see e.g.\ Section~18 in \cite{Vi0}] $M(\hat{\mathcal{R}}^n(\pi,\lambda,\tau))$ and
$M(\pi,\lambda,\tau)$ are the same elements of the moduli space.
\end{lemma}
Denote by $(\hat{\mathcal{T}}_s)_{s\in\R}$  the {\em Teichm\"uller
flow} on $\hat{\mathcal{H}}(C)$,
$$\hat{\mathcal{T}}_s(\pi,\lambda,\tau)=(\pi,e^s\lambda,e^{-s}\tau).$$
The set
$\mathcal{H}(C)=\{(\pi,\lambda,\tau)\in\hat{\mathcal{H}}(C):|\lambda|=1\}$
is a global cross-section for $(\hat{\mathcal{T}}_s)_{s\in\R}$.
Let $t_R:\hat{\mathcal{H}}(C)\to\R_+$ be defined by
$$t_R(\pi,\lambda,\tau)=-\log\left(1-\lambda_{\pi_{1-\vep}^{-1}(d)}/|\lambda|\right)\text{
whenever }(\pi,\lambda)\text{ has type }\vep.$$ If
$\hat{\mathcal{R}}(\pi,\lambda,\tau)=(\pi',\lambda',\tau')$ then
$t_R(\pi,\lambda,\tau)=-\log(|\lambda'|/|\lambda|)$ and
$|\lambda|=e^{t_R(\pi,\lambda,\tau)}|\lambda'|$. Let us consider
the Rauzy-Veech renormalization map
$\mathcal{R}:\mathcal{H}(C)\to\mathcal{H}(C)$ given by
$$\mathcal{R}=\hat{\mathcal{R}}\circ
\hat{\mathcal{T}}_{t_R(\pi,\lambda,\tau)}(\pi,\lambda,\tau)=(\pi',\lambda'/|\lambda'|,\tau'|\lambda'|).$$
Let $m$ stand for the restriction of the measure
$d\pi\,d_1\lambda\,d\tau$ to the set $\mathcal{H}(C)$, where
$d\pi$ is the counting measure on $\mathcal{P}^0_{\mathcal{A}}$,
$d_1\lambda$ is the Lebesgue measure on
$$\Lambda_{\mathcal{A}}=\{\lambda\in(\R_+\setminus\{0\})^{\mathcal{A}}:|\lambda|=1\}$$
and $d\tau$ is the Lebesgue measure on $\R^{\mathcal{A}}$.

\begin{theorem}[see Corollary~27.3 in \cite{Vi0}]\label{ergodicity}
For every Rauzy class $C\subset\mathcal{P}^0_\mathcal{A}$ the
measure $m$ is an $\mathcal{R}$-invariant ergodic conservative
measure on $\mathcal{H}(C)$.
\end{theorem}
\subsection{Different special representations of the vertical flow}
Fix $(\pi,\lambda,\tau)\in\hat{\mathcal{H}}(C)$. Recall that the
vertical flow $\mathcal{F}^i$ on $M(\pi,\lambda,\tau)$ has the
special representation over $T_{(\pi,\lambda)}$ and under
$f_h:I\to\R_+$, where
$h=h(\pi,\lambda,\tau)=-\Omega_\pi\tau\in\R_+^{\mathcal{A}}$. Let
$(\pi',\lambda',\tau')=\hat{\mathcal{R}}(\pi,\lambda,\tau)$ and
$h'=-\Omega_{\pi'}\tau'$. In view of
$\Omega_{\pi'}=\Theta(\pi,\lambda)\,\Omega_\pi\,\Theta^*(\pi,\lambda)$
(see Lemma~10.2 in \cite{Vi0}),
\[h'=-\Omega_{\pi'}\tau'=-\Omega_{\pi'}\,\Theta^{-1*}(\pi,\lambda)\tau
=-\Theta(\pi,\lambda)\,\Omega_{\pi}\tau=\Theta(\pi,\lambda)h.\]
Since $M(\pi,\lambda,\tau)$ and $M(\pi',\lambda',\tau')$ are the
same elements of the moduli space, the special flows
$T^{f_h}_{(\pi,\lambda)}$ and $T^{f_{h'}}_{(\pi',\lambda')}$ are
isomorphic. In fact, a more general result holds. We leave the
proof of the following simple lemma  to the reader.

\begin{lemma}\label{przejpot}For every interval exchange transformation $T_{(\pi,\lambda)}$
and $h\in\R_+^{\mathcal{A}}$ the special flows
$T_{(\pi,\lambda)}^{f_h}$ and
$T_{\hat{R}(\pi,\lambda)}^{f_{\Theta(\pi,\lambda)h}}$ are
isomorphic.
\end{lemma}

\section{Special flows over irrational rotations and exchanges of three intervals}
Let $A\subset \R$ be an additive subgroup. A collection of real
numbers $x_1,\ldots,x_k$ is called independent over $A$ if
$a_1x_1+\ldots+a_kx_k=0$ for $a_1,\ldots,a_k\in A$ implies
$a_1=\ldots=a_k=0$.

\begin{remark}\label{uwagagest}
Let $T_{\alpha}:[0,1)\to[0,1)$ be an ergodic rotation
$T_{\alpha}x=x+\alpha$.  Since the set $\Q+\Q\alpha$ is countable,
the set of all $(x_1,\ldots,x_k)\in\R^k$ such that
$x_1,\ldots,x_k$ are independent over $\Q+\Q\alpha$ is $G_\delta$
and dense. Denote by $DC_1$ the set of irrational numbers
$\alpha\in[0,1)$  which satisfy the following Diophantine
condition: there exists $c>0$ such that $|p-q\alpha|>c/q$ for all
$p\in\Z$ and $q\in\Z\setminus\{0\}$. Since $DC_1$ is dense in
$[0,1)$, the set $$\mathfrak{M}=\{(\alpha,\xi)\in[0,1)^2:\alpha\in
DC_1,\xi\in(\Q+\Q\alpha)\setminus(\Z+\Z\alpha)\}$$ is dense in
$[0,1)^2$.
\end{remark}
Given $\mathcal{S}=(S_t)_{t\in \R}$ a measure-preserving flow  and
$s>0$, we  denote by $\mathcal{S}^s$ the flow $(S_{st})_{t\in
\R}$. As a consequence of Theorem~1.1 in \cite{Fr-Le-Le} and
Corollary~23 in \cite{Fr-Le} we obtain the following.
\begin{proposition}\label{mild}
Let $(\alpha,\xi)\in \mathfrak{M}$ and let $a_1,a_2,a_3\in\R$ be
independent over $\Q+\Q\alpha$ and such that
$f=a_1+a_2\chi_{[0,\xi)}+a_3\chi_{[0,1-\alpha)}>0$. Then the
special flow built over  $T_\alpha$ and under the roof function
$f$ is mildly mixing. Moreover, the flows $T_\alpha^f$ and
$(T_\alpha^f)^s$ are not isomorphic for all positive $s\neq 1$.
\end{proposition}

Let $\mathcal{A}=\{a,b,c\}$,
\[\pi_s=\left(\begin{array}{cccc}a&b&c\\c&b&a\end{array}\right),
\;\;\pi_l=\left(\begin{array}{cccc}a&b&c\\b&c&a\end{array}\right),\;\;\pi_r=\left(\begin{array}{cccc}a&b&c\\c&a&b\end{array}\right)
\text{ and }
\]

\[\Lambda^l_{\mathcal{A}}=\Lambda^r_{\mathcal{A}}=\Lambda_{\mathcal{A}},\;\;\Lambda^0_{\mathcal{A}}=\{\lambda\in\Lambda_{\mathcal{A}}:\lambda_a<\lambda_c\},
\;\;\Lambda^1_{\mathcal{A}}=\{\lambda\in\Lambda_{\mathcal{A}}:\lambda_a>\lambda_c\}.\]
 Let us consider four functions
$\rho_\gamma:\Lambda^\gamma_{\mathcal{A}}\to[0,1]^2$,
$\gamma\in\{l,r,0,1\}$ defined by
\[\rho_l(x_a,x_b,x_c)=(1-x_a,1-x_c), \;\;\rho_r(x_a,x_b,x_c)=(x_c,x_a),\]
\[\rho_0(x_a,x_b,x_c)=\left(\frac{x_c-x_a}{1-x_a},\frac{x_a}{1-x_a}\right),\;\;\rho_1(x_a,x_b,x_c)=\left(\frac{1-x_a}{1-x_c},\frac{x_a}{1-x_c}\right).\]
Obviously,
$\rho_\gamma:\Lambda^\gamma_{\mathcal{A}}\to\rho(\Lambda^\gamma_{\mathcal{A}})$
is a $C^\infty$--diffeomorphism and
$\rho_\gamma(\Lambda^\gamma_{\mathcal{A}})\subset[0,1]^2$ is open
for $\gamma=l,r,0,1$. Let
\[\gamma({\pi},{\lambda})=\left\{
\begin{array}{ccl}
l&\text{ if }& {\pi}=\pi_l\\
r&\text{ if }& {\pi}=\pi_r\\
0&\text{ if }& {\pi}=\pi_s \text{ and
}\lambda\in \Lambda^0_{\mathcal{A}}\\
1&\text{ if }& {\pi}=\pi_s \text{ and
}\lambda\in \Lambda^1_{\mathcal{A}}.\\
\end{array} \right.\]
Let us consider $\rho:\mathcal{P}^0_{\mathcal{A}}\times
\Lambda_{\mathcal{A}}\to[0,1]^2$ given by
$\rho(\pi,\lambda)=\rho_{\gamma({\pi},{\lambda})}(\lambda)$. We
will use the notation $(\alpha(\pi,\lambda),\xi(\pi,\lambda))$ for
$\rho(\pi,\lambda)$.
\begin{corollary}\label{wniosek}
For every $\pi\in\mathcal{P}^0_{\mathcal{A}}$,
$\lambda\in\Lambda_{\mathcal{A}}$ and $h\in\R_+^{\mathcal{A}}$ if
$\rho(\pi,\lambda)\in\mathfrak{M}$ and $h_1,h_2,h_3$ are
independent over $\Q+\Q\alpha(\pi,\lambda)$ then the special flow
$T_{(\pi,\lambda)}^{f_h}$ is mildly mixing.
\end{corollary}

\begin{proof}We will prove the claim for the cases $r$ and $0$. In the
remaining cases  the proof is similar, and we leave it to the
reader.

Suppose that $\pi=\pi_{r}$, $\lambda\in\Lambda_{\mathcal{A}}$ and
$h\in\R_+^{\mathcal{A}}$. Then $T_{(\pi,\lambda)}$ is isomorphic
to the circle rotation by $\lambda_c=\alpha(\pi,\lambda)$ and
\begin{eqnarray*}f_h&=&h_a+(h_b-h_a)\chi_{[0,\lambda_a)}+(h_c-h_b)\chi_{[0,1-\lambda_c)}\\
&=&h_a+(h_b-h_a)\chi_{[0,\xi(\pi,\lambda))}+(h_c-h_b)\chi_{[0,1-\alpha(\pi,\lambda))}.\end{eqnarray*}
Suppose that
$\rho(\pi,\lambda)=(\alpha(\pi,\lambda),\xi(\pi,\lambda))\in\mathfrak{M}$
and $h_a,h_b,h_c$ are independent over $\Q+\Q\alpha(\pi,\lambda)$.
Then $h_a,h_b-h_a,h_c-h_b$ are independent over
$\Q+\Q\alpha(\pi,\lambda)$. Now Proposition~\ref{mild} implies the
mild mixing of $T_{(\pi,\lambda)}^{f_h}$.

Next, suppose that $\pi=\pi_{s}$,
$\lambda\in\Lambda^0_{\mathcal{A}}$, $h\in\R_+^{\mathcal{A}}$.
Then $(\pi,\lambda)$ has type $0$. Let
$(\pi',\lambda')=\hat{R}(\pi,\lambda)$ and
$h'=\Theta(\pi,\lambda)h$. Thus $\pi'=\pi_r$,
\[\lambda'=\Theta(\pi,\lambda)^{-1*}\lambda=(\lambda_a,\lambda_b,\lambda_c-\lambda_a)\text{ and }
h'=(h_a+h_c,h_b,h_c).\] By Lemma~\ref{przejpot}, the special flows
$T_{(\pi,\lambda)}^{f_h}$ and $T_{(\pi_r,\lambda')}^{f_{h'}}$ are
isomorphic. Note that
\[
\rho(\pi_r,\lambda'/|\lambda'|)=\rho_r(\lambda'/|\lambda'|)
=\rho_r\left(\frac{\lambda_a}{1-\lambda_a},\frac{\lambda_b}{1-\lambda_a},\frac{\lambda_c-\lambda_a}{1-\lambda_a}\right)
=\rho_0(\lambda)=\rho(\pi,\lambda).\]
 Suppose that
$\rho(\pi,\lambda)\in\mathfrak{M}$ and $h_a,h_b,h_c$ are
independent over $\Q+\Q\alpha(\pi,\lambda)$. It follows that
$h_a+h_c,h_b,h_c$ are independent over $\Q+\Q\alpha(\pi,\lambda)$.
Since $\rho(\pi_r,\lambda'/|\lambda'|)=\rho(\pi,\lambda)$ and
$\alpha(\pi_r,\lambda'/|\lambda'|)=\alpha(\pi,\lambda)$, we have
$\rho(\pi_r,\lambda'/|\lambda'|)\in\mathfrak{M}$ and
$h'_a,h'_b,h'_c$ are independent over
$\Q+\Q\alpha(\pi_r,\lambda'/|\lambda'|)$. By the first part of the
proof, the special flow $T_{(\pi_r,\lambda'/|\lambda'|)}^{f_{h'}}$
is mildly mixing. It follows that $T_{(\pi_r,\lambda')}^{f_{h'}}$,
and hence $T_{(\pi,\lambda)}^{f_h}$, is mildly mixing.
\end{proof}
\section{Mild mixing of vertical flows}
Let $\|x\|=\sum_{\alpha\in\mathcal{A}}|x_\alpha|$ for every
$x\in\R^{\mathcal{A}}$. For every  matrix
$B=[b_{\alpha\,\beta}]_{\alpha\,\beta\in\mathcal{A}}$ with
positive entries let
$\nu(B)=\max_{\alpha,\beta,\gamma\in\mathcal{A}}{b_{\alpha\,\beta}}/{b_{\alpha\,\gamma}}.$
 Then
\begin{equation}\label{nierlip}
\left\|\frac{B\lambda'}{|B\lambda|}-\frac{B\lambda}{|B\lambda|}\right\|\leq
\nu(B)^2\|\lambda-\lambda'\|\text{ for all
}\lambda,\lambda'\in\Lambda_{\mathcal{A}}.
\end{equation}

We will denote by $\Arg:\C\setminus\{0\}\to(-\pi,\pi]$ the
principal argument function. Recall that for every $z_1,z_2$ with
nonnegative real parts we have $\Arg(z_1+z_2)=\Arg z_1+\Arg z_2$
and $\Arg(\overline{z}_1)=-\Arg(z_1)$.

 Let
$\mathcal{A}=\{1,\ldots,d\}$, $d\geq 4$. Assume that
$\overline{\pi}\in\mathcal{P}^0_{\mathcal{A}}$ is a standard pair
such  that $\bar{\pi}_0$ is the identity. Let
\[Z(\bar{\pi})=\{(\bar{\pi},\lambda,\tau):
\lambda_1=\ldots=\lambda_{d-3}=0,(\bar{\pi},\lambda,\tau)\in
\hat{\mathcal{H}}(C),\tau\in\mathcal{T}^+_{\pi,\lambda}\}.\]

\begin{lemma}\label{prostogen}
The set $\{M(\bar{\pi},\lambda,\tau):(\bar{\pi},\lambda,\tau)\in
Z(\bar{\pi})\}$ is dense in $M(\hat{\mathcal{H}}(C))$.
\end{lemma}

\begin{proof} The proof consists of four steps.
In the first step, using the extended Veech-Rauzy induction,  for
almost every $(\pi,\lambda,\tau)\in\mathcal{H}(C)$ we find a
representation of $M(\pi,\lambda,\tau)$ which is given by
$(\bar{\pi},\lambda^{(n)},\tau^{(n)})=\hat{\mathcal{R}}^{k_n}(\pi,\lambda,\tau)$
so that the first $d-3$ sides of the polygon
$P(\bar{\pi},\lambda^{(n)},\tau^{(n)})$ are almost parallel. In
the second step, $(\bar{\pi},\lambda^{(n)},\tau^{(n)})$ is
perturbed to get $(\bar{\pi},\lambda^{p\,(n)},\tau^{(n)})$ such
that the first $d-3$ sides of the polygon
$P(\bar{\pi},\lambda^{p\,(n)},\tau^{(n)})$ are parallel. To
describe this perturbation we will need two auxiliary substeps
passing by
\[(\tilde{\lambda}^{(n)},\tilde{\tau}^{(n)})=(\lambda^{(n)}/|\lambda^{(n)}|,\tau^{(n)}|\lambda^{(n)}|),\;
(\tilde{\lambda}^{p\,(n)},\tilde{\tau}^{(n)})=(\lambda^{p\,(n)}/|\lambda^{(n)}|,\tau^{(n)}|\lambda^{(n)}|).\]
 In the third step, $(\bar{\pi},\lambda^{p\,(n)},\tau^{(n)})$ is
rotated by an angle $\theta_n$ ($\theta_n\to 0$ as $n\to\infty$)
to obtain
$(\bar{\pi},\lambda^{r\,(n)},\tau^{r\,(n)})\in\hat{\mathcal{H}}(C)$
so that the first $d-3$ sides of the polygon
$P(\bar{\pi},\lambda^{r\,(n)},\tau^{r\,(n)})$ are vertical, hence
$(\bar{\pi},\lambda^{r\,(n)},\tau^{r\,(n)})\in Z(\bar{\pi})$. In
the final step, we show that
$M(\bar{\pi},\lambda^{r\,(n)},\tau^{r\,(n)})\to
M(\pi,\lambda,\tau)$. In order to do this, applying the inverse of
the renormalization, we prove that
$$(\pi,\lambda^{b\,(n)},\tau)=\hat{\mathcal{R}}^{-k_n}(\bar{\pi},\lambda^{p\,(n)},\tau^{(n)})\to
(\pi,\lambda,\tau).$$ In view of (\ref{obrot}), it follows that
\[M(\bar{\pi},\lambda^{r\,(n)},\tau^{r\,(n)})=\mathcal{U}_{\theta_n}M(\bar{\pi},\lambda^{p\,(n)},\tau^{(n)})=
\mathcal{U}_{\theta_n}M(\pi,\lambda^{b\,(n)},\tau)\to
M(\pi,\lambda,\tau).\]

{\bf Step 1.} Let $A_n$ stand for the set of
$(\bar{\pi},\lambda,\tau)\in\mathcal{H}(C)$ such that
\begin{equation}\label{war1}
\lambda_j>0,\;\;\;\frac{\tau_1}{\lambda_1}>1,\;\;\;-\sum_{\bar{\pi}_1(k)\leq
j}\tau_k>\frac{\lambda_1}{\tau_1} \text{ for }j=1,\ldots,d,
\end{equation}
\begin{equation}\label{war2}
\left|\frac{\tau_1}{\lambda_1}-\frac{\tau_j}{\lambda_j}\right|<\frac{1}{n}\text{
for }j=2,\ldots,d-3,\;\;\tau_j<0\text{ for }j=d-2,d-1,d,
\end{equation}
\begin{equation}\label{war3}
\frac{\lambda_1}{\tau_1}\sum_{k=1}^{d-3}\tau_k+\lambda_{d-2}+\lambda_{d-1}<1.
\end{equation}
 Note that if $(\bar{\pi},\lambda,\tau)\in A_n$ then the first $d-3$ sides of the polygon
$P(\bar{\pi},\lambda,\tau)$ are almost parallel. Setting
$\lambda=(1/d,\ldots,1/d)$, $\tau_j=2/d$ for $j=1,\ldots,d-3$,
$\tau_{d-2}=\tau_{d-1}=-1/2d$ and $\tau_d=-3$, since
$\overline{\pi}$ is a standard pair, we get
$(\bar{\pi},\lambda,\tau)\in A_n$. It follows that $A_n$ is a
nonempty open subset of $\mathcal{H}(C)$ and hence $m(A_n)>0$.

By Theorem~\ref{ergodicity}, using standard Veech arguments (see
\cite[Ch.\ 3]{Ve2}), there exists $\Gamma>0$ and a measurable
subset $B\subset\mathcal{H}(C)$ such that $m(B^c)=0$ and for every
$(\pi,\lambda,\tau)\in B$ there exists a sequence $k_n\to+\infty$
such that $\mathcal{R}^{k_n}(\pi,\lambda,\tau)\in A_n$,
$\Theta^{(k_n)}(\pi,\lambda)$ has positive entries and
$\nu(\Theta^{(k_n)}(\pi,\lambda)^*)\leq \Gamma$. Let
$(\bar{\pi},\lambda^{(n)},\tau^{(n)})=\hat{\mathcal{R}}^{k_n}(\pi,\lambda,\tau)$
and
$$(\bar{\pi},\tilde{\lambda}^{(n)},\tilde{\tau}^{(n)})=\mathcal{R}^{k_n}(\pi,\lambda,\tau)=(\bar{\pi},\lambda^{(n)}/|\lambda^{(n)}|,\tau^{(n)}|\lambda^{(n)}|).$$
Since  $(\bar{\pi},\tilde{\lambda}^{(n)},\tilde{\tau}^{(n)})\in
A_n$, we have
\begin{equation}\label{jakies1}
\tilde{\lambda}^{(n)}_j>0,\;\;\;\frac{\tilde{\tau}^{(n)}_1}{\tilde{\lambda}^{(n)}_1}>1,\;\;\;-\sum_{\bar{\pi}_1(k)\leq
j}\tilde{\tau}^{(n)}_k>\frac{\tilde{\lambda}^{(n)}_1}{\tilde{\tau}^{(n)}_1}\text{
for  }j=1,\ldots,d,
\end{equation}
\begin{equation}\label{jakies2}
\left|\frac{\tilde{\tau}^{(n)}_1}{\tilde{\lambda}^{(n)}_1}-\frac{\tilde{\tau}^{(n)}_j}{\tilde{\lambda}^{(n)}_j}\right|<\frac{1}{n}\text{
for }j=2,\ldots,d-3,\;\;\tilde{\tau}^{(n)}_j<0\text{ for
}j=d-2,d-1,d,
\end{equation}
\begin{equation}\label{jakies3}
\frac{\tilde{\lambda}^{(n)}_1}{\tilde{\tau}^{(n)}_1}\sum_{k=1}^{d-3}\tilde{\tau}^{(n)}_k+\tilde{\lambda}^{(n)}_{d-2}+\tilde{\lambda}^{(n)}_{d-1}<1.
\end{equation}
  Moreover,
$\tau^{(n)}\in H^+_{\bar{\pi}}$. From (\ref{jakies1}) and
(\ref{jakies2}), we have $\tilde{\tau}^{(n)}_j>0$ for
$j=1,\ldots,d-3$.

{\bf Step 2.} Let us consider
$\tilde{\lambda}^{p\,(n)}\in\R^{\mathcal{A}}$ with
\[ \tilde{\lambda}^{p\,(n)}_j=\left\{\begin{array}{cll}
\frac{\tilde{\lambda}^{(n)}_1}{\tilde{\tau}^{(n)}_1}\tilde{\tau}^{(n)}_j&\text{
if }&j=1,\ldots,d-3\\
\tilde{\lambda}^{(n)}_j&\text{ if }&j=d-2,d-1\\
1-\sum_{j=1}^{d-1}\tilde{\lambda}^{p\,(n)}_j&\text{ if }& j=d.
\end{array}\right.\]
It follows from (\ref{jakies3}) that
$\tilde{\lambda}^{p\,(n)}\in\Lambda_{\mathcal{A}}$. Since
$\left|\frac{\tilde{\tau}^{(n)}_j}{\tilde{\lambda}^{p\,(n)}_j}-\frac{\tilde{\tau}^{(n)}_j}
{\tilde{\lambda}^{(n)}_j}\right|<\frac{1}{n}$ and
$\frac{\tilde{\lambda}^{(n)}_1}{\tilde{\tau}^{(n)}_1}<1$, we
obtain
\[|\tilde{\lambda}^{p\,(n)}_j-\tilde{\lambda}^{(n)}_j|<
\frac{1}{n}\frac{\tilde{\lambda}^{p\,(n)}_j
\tilde{\lambda}^{(n)}_j}{\tilde{\tau}^{(n)}_j}=\frac{\tilde{\lambda}^{(n)}_j}{n}\frac{\tilde{\lambda}^{(n)}_1
}{\tilde{\tau}^{(n)}_1}<\frac{\tilde{\lambda}^{(n)}_j}{n}\text{
for }j=1,\ldots,d-3,\] and hence
$|\tilde{\lambda}^{p\,(n)}_d-\tilde{\lambda}^{(n)}_d|<1/n$.
Therefore,
$\|\tilde{\lambda}^{p\,(n)}-\tilde{\lambda}^{(n)}\|<2/n$.
Moreover, by (\ref{jakies1}),
\begin{equation}\label{zredjed}
-\frac{\sum_{\bar{\pi}_1(k)\leq
j}\tilde{\tau}^{(n)}_k}{\sum_{\bar{\pi}_1(k)\leq j}
\tilde{\lambda}^{p\,(n)}_k}>-\sum_{\bar{\pi}_1(k)\leq
j}\tilde{\tau}^{(n)}_k>\frac{\tilde{\lambda}^{(n)}_1}{\tilde{\tau}^{(n)}_1}\text{
for }j=1,\ldots,d.
\end{equation}
Let $\lambda^{p\,(n)}=|\lambda^{(n)}|\tilde{\lambda}^{p\,(n)}$. As
$\tau^{(n)}\in H^+_{\bar{\pi}}$, we have
$(\bar{\pi},\lambda^{p\,(n)},\tau^{(n)})\in\hat{\mathcal{H}}(C)$.
Since $\tau^{(n)}=\tilde{\tau}^{(n)}/|\lambda^{(n)}|$, by
(\ref{jakies1}), (\ref{jakies2}) and (\ref{zredjed}), we obtain
\begin{equation}\label{slope}
\frac{\tau^{(n)}_j}{\lambda^{p\,(n)}_j}=\frac{\tau^{(n)}_1}{\lambda^{(n)}_1}>\frac{1}{|\lambda^{(n)}|^2}\text{
for } j=1,\ldots,d-3,\;\;\tau^{(n)}_j<0\text{ for }j=d-2,d-1,d
\end{equation}
and
\begin{equation}\label{jeszczeslope}
\begin{split}-\frac{\sum_{\bar{\pi}_1(k)\leq
j}{\tau}^{(n)}_k}{\sum_{\bar{\pi}_1(k)\leq j}
{\lambda}^{p\,(n)}_k}&=-\frac{\sum_{\bar{\pi}_1(k)\leq
j}\tilde{\tau}^{(n)}_k}{\sum_{\bar{\pi}_1(k)\leq j}
\tilde{\lambda}^{p\,(n)}_k}\frac{1}{|\lambda^{(n)}|^2}\\&>
\frac{\tilde{\lambda}^{(n)}_1}{\tilde{\tau}^{(n)}_1}\frac{1}{|\lambda^{(n)}|^2}
=\frac{{\lambda}^{(n)}_1}{{\tau}^{(n)}_1}\frac{1}{|\lambda^{(n)}|^4}>\frac{{\lambda}^{(n)}_1}{{\tau}^{(n)}_1}
\end{split}\end{equation} for $j=1,\ldots,d$.

{\bf Step 3.} Let
\[\theta_n=\pi/2-\Arg (\lambda^{(n)}_1+i\tau^{(n)}_1)=\Arg(\tau_1^{(n)}+i\lambda_1^{(n)})>0.\] Since
$|\lambda^{(n)}|\to 0$, by (\ref{slope}), we obtain $\theta_n\to
0$. Let $$\lambda^{r\,(n)}+i\tau^{r\,(n)}=e^{ i\theta_n
}(\lambda^{p\,(n)}+i\tau^{(n)}).$$ In this step we will prove that
$(\bar{\pi},\lambda^{r\,(n)},\tau^{r\,(n)})\in Z(\bar{\pi})$. As
$\Arg(\lambda^{p\,(n)}_j+i\tau^{(n)}_j)=\Arg(\lambda^{(n)}_1+i\tau^{(n)}_1)$
for $j=1,\ldots,d-3$ and
$-\pi/2<\Arg(\lambda^{p\,(n)}_j+i\tau^{(n)}_j)<0$ for
$j=d-2,d-1,d$, we have
\begin{equation*}\Arg({\lambda}_j^{r\,(n)}+i{\tau}_j^{r\,(n)})=\Arg(\lambda^{p\,(n)}_j+i\tau^{(n)}_j)+\pi/2-\Arg
(\lambda^{(n)}_1+i\tau^{(n)}_1)=\pi/2
\end{equation*} for
$j=1,\ldots,d-3$ and
\[-\pi/2<\Arg({\lambda}^{r\,(n)}_j+i{\tau}^{r\,(n)}_j)=\Arg(\lambda^{p\,(n)}_j+i\tau^{(n)}_j)+\pi/2-\Arg
(\lambda^{(n)}_1+i\tau^{(n)}_1)<\pi/2\] for $j=d-2,d-1,d$. It
follows that
\begin{equation}\label{zera}{\lambda}^{r\,(n)}_j=0,\;{\tau}^{r\,(n)}_j>0\text{ for }j=1,\ldots,d-3\text{
and }{\lambda}^{r\,(n)}_j>0\text{ for }j=d-2,d-1,d.
\end{equation} Since
$\tau^{(n)}\in H^+_{\bar{\pi}}$, we have
$0<\Arg(\sum_{k=1}^j{\lambda}_k^{p\,(n)}+i{\tau}_k^{(n)})<\pi/2$,
and hence
\[\Arg(\sum_{k=1}^j{\lambda}^{r\,(n)}_k+i{\tau}^{r\,(n)}_k)=\Arg(\sum_{k=1}^j{\lambda}_k^{p\,(n)}+i{\tau}_k^{(n)})+\theta_n>
\Arg(\sum_{k=1}^j{\lambda}_k^{p\,(n)}+i{\tau}_k^{(n)})>0\] for
$j=1,\ldots,d-1$. Therefore, $\sum_{k=1}^j{\tau}^{r\,(n)}_j>0$ for
$j=1,\ldots,d-1$. By (\ref{jeszczeslope}),
\begin{eqnarray*}0&>&\Arg(\sum_{\bar{\pi}_1(k)\leq
j}{\lambda}_k^{p\,(n)}+i{\tau}_k^{(n)})+\Arg(\tau_1^{(n)}+i\lambda_1^{(n)})\\
&=& \Arg(\sum_{\bar{\pi}_1(k)\leq
j}{\lambda}_k^{p\,(n)}+i{\tau}_k^{(n)})+\theta_n=\Arg(\sum_{\bar{\pi}_1(k)\leq
j}{\lambda}^{r\,(n)}_k+i{\tau}^{r\,(n)}_k)
\end{eqnarray*}
and hence $\sum_{\bar{\pi}_1(k)\leq j}{\tau}^{r\,(n)}_k<0$ for all
$j=1,\ldots,d$. Therefore,
$(\bar{\pi},\lambda^{r\,(n)},\tau^{r\,(n)})\in\hat{\mathcal{H}}(C)$.
In view of (\ref{zera}), it follows that
$(\bar{\pi},\lambda^{r\,(n)},\tau^{r\,(n)})\in Z(\overline{\pi})$.

{\bf Step 4.}   Let
\[\lambda^{b\,(n)}=\Theta^{(k_n)}(\pi,\lambda)^*{\lambda}^{p\,(n)}=
|\lambda^{(n)}|\Theta^{(k_n)}(\pi,\lambda)^*\tilde{\lambda}^{p\,(n)}=
\frac{\Theta^{(k_n)}(\pi,\lambda)^*\tilde{\lambda}^{p\,(n)}}{|\Theta^{(k_n)}(\pi,\lambda)^*\tilde{\lambda}^{(n)}|}.\]
Since $\nu(\Theta^{(k_n)}(\pi,\lambda)^*)\leq \Gamma$, by
(\ref{nierlip}),
\[\|\lambda^{b\,(n)}-\lambda\|=\left\|\frac{\Theta^{(k_n)}(\pi,\lambda)^*\tilde{\lambda}^{p\,(n)}}{|\Theta^{(k_n)}(\pi,\lambda)^*\tilde{\lambda}^{(n)}|}-
\frac{\Theta^{(k_n)}(\pi,\lambda)^*\tilde{\lambda}^{(n)}}{|\Theta^{(k_n)}(\pi,\lambda)^*\tilde{\lambda}^{(n)}|}\right\|\leq
\Gamma^2\|\tilde{\lambda}^{p\,(n)}-\tilde{\lambda}^{(n)}\|\leq\frac{2\Gamma^2}{n}.
\]
Moreover, by Remark~\ref{powrot},
\[\hat{R}^{k_n}(\pi,\lambda^{b\,(n)})=(\bar{\pi},\Theta^{(k_n)}(\pi,\lambda)^{-1*}\lambda^{b\,(n)})=
(\bar{\pi},{\lambda}^{p\,(n)})\] and
\[\hat{\mathcal{R}}^{k_n}(\pi,\lambda^{b\,(n)},\tau)=(\bar{\pi},\Theta^{(k_n)}(\pi,\lambda)^{-1*}\lambda^{b\,(n)},\Theta^{(k_n)}(\pi,\lambda)^{-1*}\tau)=
(\bar{\pi},{\lambda}^{p\,(n)},\tau^{(n)}).\] Hence
$M(\bar{\pi},\lambda^{p\,(n)},\tau^{(n)})=M(\pi,\lambda^{b\,(n)},\tau)$.
In view of (\ref{obrot}), it follows that
\[M(\bar{\pi},\lambda^{r\,(n)},\tau^{r\,(n)})=\mathcal{U}_{\theta_n}M(\bar{\pi},\lambda^{p\,(n)},\tau^{(n)})=
\mathcal{U}_{\theta_n}M(\pi,\lambda^{b\,(n)},\tau).\] Since
$\|\lambda^{b\,(n)}-\lambda\|<2/n$ and $\theta_n\to 0$, by the
continuity of the map $M$ (see Proposition~\ref{contm}) and the
flow $(\mathcal{U}_s)_{s\in\R}$, it follows that
$M(\bar{\pi},\lambda^{r\,(n)},\tau^{r\,(n)})\to
M(\pi,\lambda,\tau)$ in the moduli space for every
$(\pi,\lambda,\tau)\in B\subset \mathcal{H}(C)$. Furthermore, for
every real $s>0$ we have
$M(\bar{\pi},s\lambda^{r\,(n)},\tau^{r\,(n)})\to
M({\pi},s\lambda,\tau)$.

Let
$\widetilde{B}=\{({\pi},s\lambda,\tau)\in\hat{\mathcal{H}}(C):(\pi,\lambda,\tau)\in
B \}$. Since  the topological support of $m$ is $\mathcal{H}(C)$
and $m(B^c)=0$, the set $B$ is dense in $\mathcal{H}(C)$, and
hence $\widetilde{B}$ is dense in $\hat{\mathcal{H}}(C)$. As
$(\bar{\pi},s\lambda^{r\,(n)},\tau^{r\,(n)})\in Z(\bar{\pi})$, it
follows that $M(Z(\bar{\pi}))$ is dense in
$M(\hat{\mathcal{H}}(C))$.
\end{proof}

\begin{lemma}\label{niezaleznosc}
Suppose that $\mathcal{A}=\{1,\ldots,d\}$ with $d\geq 4$ and
$\bar{\pi}\in\mathcal{P}^*_{\mathcal{A}}$ is a standard pair such
that $\bar{\pi}_0=id$. Assume that $\tau_1,\ldots,\tau_d$ are
independent over an additive subgroup $A\subset\R$. Let
$h=-\Omega_{\overline{\pi}} \tau$. Then $h_{d-2},h_{d-1},h_d$ are
also independent over $A$.
\end{lemma}
\begin{proof}Suppose that $a_1h_{d-2}+a_2h_{d-1}+a_3h_d=0$ and $a_1,a_2,a_3\in A$. Since
$\bar{\pi}\in\mathcal{P}^{*}_{\mathcal{A}}$, we have
$\bar{\pi}_1(d-1)\neq\bar{\pi}_1(d-2)+1$. Hence there exists
$1<s<d-1$ such that $\Omega_{s\,(d-2)}\neq \Omega_{s\,(d-1)}$.
Since $\overline{\pi}$ is a standard pair,
\begin{eqnarray*}h_{d-2}&=&\tau_1+\ldots+\Omega_{s\,(d-2)}\tau_s+\ldots-\tau_d\\
h_{d-1}&=&\tau_1+\ldots+\Omega_{s\,(d-1)}\tau_s+\ldots-\tau_d\\
h_{d}&=&\tau_1+\ldots+\tau_s+\ldots+\tau_{d-1},
\end{eqnarray*}
and hence
\[(a_1+a_2+a_3)\tau_1+\ldots+(a_1\Omega_{s\,(d-2)}+a_2\Omega_{s\,(d-1)}+a_3)\tau_s+\ldots+(-a_1 -a_2)\tau_d=0.\]
Therefore
\[a_1+a_2+a_3=a_1\Omega_{s\,(d-2)}+a_2\Omega_{s\,(d-1)}+a_3=a_1+a_2=0.\]
Since $\Omega_{s\,(d-2)}\neq \Omega_{s\,(d-1)}$, it follows that
 $a_1=a_2=a_3=0$.
\end{proof}
\begin{theorem}
\label{milddensegen} If $g\geq 2$ then for every stratum
$\mathcal{H}_g(m_1,\ldots,m_{\kappa})$ there exists a dense subset
 $\mathcal{H}_{mm}\subset\mathcal{H}_g(m_1,\ldots,m_{\kappa})$ such that for every
$(M,\omega)\in\mathcal{H}_{mm}$ its vertical flow is mildly
mixing.
\end{theorem}

\begin{proof}
By Proposition~\ref{realizacja} and Remark~\ref{gestpow}, there
exists a finite family $\mathcal{C}$ of Rauzy classes in
$\mathcal{P}^*_{\mathcal{A}}$ ($\#\mathcal{A}=d=2g+\kappa-1\geq
4$) such that $\bigcup_{C\in\mathcal{C}}M(\hat{\mathcal{H}}(C))$
is dense in $\mathcal{H}_g(m_1,\ldots,m_{\kappa})$.

Let $\mathcal{A}=\{1,\ldots,d\}$. In view of
Proposition~\ref{Rauzy} and Lemma~\ref{prostogen}, it suffices to
show that for every $\bar{\pi}$  standard pair in $C$ such that
$\bar{\pi}_0=id$ and for every $(\bar{\pi},\lambda,\tau)\in
Z(\bar{\pi})$ there exists a sequence
$\{(\bar{\pi},\lambda^n,\tau^n)\}_{n\in\N}$ in $Z(\overline{\pi})$
such that $(\lambda^n,\tau^n)\to (\lambda,\tau)$ and the vertical
flow for $M(\bar{\pi},\lambda^n,\tau^n)$ is mildly mixing. Without
loss of generality we can assume that $|\lambda|=1$. Moreover, we
can also assume that $\lambda_{d-2}$,
$\lambda_{d-1}$,$\lambda_{d}$ are positive and
$\lambda_{d-2}\neq\lambda_d$, because the set of all
$(\bar{\pi},\lambda,\tau)\in Z(\bar{\pi})$ satisfying this
condition is dense in $Z(\bar{\pi})$.

Suppose that $(\bar{\pi},\lambda,\tau)$ is an element of
$Z(\bar{\pi})$ such that $\lambda_{d-2}$,
$\lambda_{d-1}$,$\lambda_{d}$ are positive and
$\lambda_{d-2}\neq\lambda_d$. Let
$h=h(\tau)=-\Omega_{\overline{\pi}} \tau$. Since
$(\bar{\pi},\lambda,\tau)\in Z(\bar{\pi})$ and $\overline{\pi}$ is
a standard pair, by Lemma~\ref{spec}, the vertical flow for
$M(\bar{\pi},\lambda,\tau)$
 is isomorphic to the special flow
$T_{(\tilde{\pi},\tilde{\lambda})}^{f_{\tilde{h}}}$, where
$T_{(\tilde{\pi},\tilde{\lambda})}$ is an exchange on three
intervals such that $\tilde{\pi}\in \mathcal{P}^0_{\{d-2,d-1,d\}}$
is equal to
$$\pi_r=\left(\begin{array}{cccc}d-2&d-1&d\\d&d-2&d-1\end{array}\right)\text{
or }
\pi_s=\left(\begin{array}{cccc}d-2&d-1&d\\d&d-1&d-2\end{array}\right),$$
$\tilde{\lambda}=(\lambda_{d-2},\lambda_{d-1},\lambda_d)\in\Lambda_{\{d-2,d-1,d\}}$
and $f_{\tilde{h}}$ is determined by
$\tilde{h}=(h_{d-2},h_{d-1},h_d)$. Let
$\gamma=\gamma(\tilde{\pi},\tilde{\lambda})$. Since
$\rho_\gamma:\Lambda^\gamma_{\{d-2,d-1,d\}}\to
\rho_\gamma(\Lambda^\gamma_{\{d-2,d-1,d\}})\subset[0,1]^2$ is a
diffeomorphism and $\mathfrak{M}$ is  dense in $[0,1]^2$, we can
find a sequence
$\{(\lambda^n_{d-2},\lambda^n_{d-1},\lambda^n_d)\}_{n\in\N}$ in
$\Lambda^\gamma_{\{d-2,d-1,d\}}$ such that
$$(\lambda^n_{d-2},\lambda^n_{d-1},\lambda^n_d)\to\tilde{\lambda}
\text{ and }
\rho(\tilde{\pi},(\lambda^n_{d-2},\lambda^n_{d-1},\lambda^n_d))=\rho_\gamma(\lambda^n_{d-2},\lambda^n_{d-1},\lambda^n_d)\in
\mathfrak{M}.$$ Setting
$\lambda^n=(0,\ldots,0,\lambda^n_{d-2},\lambda^n_{d-1},\lambda^n_d)\in\Lambda_\mathcal{A}$,
we have
$\widetilde{\lambda}^n=(\lambda^n_{d-2},\lambda^n_{d-1},\lambda^n_d)$
and $\lambda^n\to\lambda$. Since
$\mathcal{T}^+_{\overline{\pi},\lambda^n}$ is open, there exists a
sequence $\{\tau^n\}_{n\in\N}$ such that
$\tau^n\in\mathcal{T}^+_{\overline{\pi},\lambda^n}$,
$\tau^n\to\tau$ and $\tau^n_1,\ldots,\tau^n_d$ are independent
over $\Q+\Q\alpha(\tilde{\pi},\widetilde{\lambda}^n)$. In view of
Lemma~\ref{niezaleznosc}, $h_{d-2}(\tau^n)$, $ h_{d-2}(\tau^n)$,
$h_d(\tau^n)$ are also independent over
$\Q+\Q\alpha(\tilde{\pi},\tilde{\lambda}^n)$. By
Corollary~\ref{wniosek}, it follows that
$T_{(\tilde{\pi},\tilde{\lambda}^n)}^{f_{\tilde{h}(\tau^n)}}$ is
mildly mixing. Consequently, the vertical flow of
$M(\overline{\pi},\lambda^n,\tau^n)$ is also mildly mixing. As
$(\overline{\pi},\lambda^n,\tau^n)\in Z(\overline{\pi})$ and
$(\lambda^n,\tau^n)\to(\lambda,\tau)$, the theorem follows.
\end{proof}
\begin{corollary} If $g\geq 2$ then the set of Abelian
differentials in $\mathcal{H}_g$ for which the vertical flow is
mildly mixing is dense in $\mathcal{H}_g$.
\end{corollary}

\section{Measure-theoretical equivalence of Abelian differentials and some orbits of the Teichm\"uller flow}
\begin{definition}
Two Abelian differentials $(M,\omega)$ and  $(M',\omega')$ are
measure-theoretical isomorphic if there exists a
measure-preserving invertible map $\psi:(M,\omega)\to(M',\omega')$
such that
$\psi\circ\mathcal{F}_s^{\omega,\theta}=\mathcal{F}_s^{\omega',\theta}\circ\psi$
for every $\theta\in S^1$ and $s\in\R$.
\end{definition}

For every stratum $\mathcal{H}_g(m_1,\ldots,m_\kappa)$ let
$(\mathcal{T}_t)_{t\in\R}$ stand for the Teichm\"uller geodesic
flow on $\mathcal{H}_g(m_1,\ldots,m_\kappa)$. As a consequence of
results from previous sections we obtain the following.

\begin{theorem}
If $g\geq 2$ then there exists a dense subset
$\mathcal{H}'\subset\mathcal{H}_g(m_1,\ldots,m_\kappa)$ such that
for every $(M,\omega)\in\mathcal{H}'$ the Abelian differentials
$(M,\omega)$ and $\mathcal{T}_s(M,\omega)$ are not
measure-theoretically equivalent for every real $s\neq 0$.
\end{theorem}

\begin{proof}
By Proposition~\ref{mild} and the proof of
Theorem~\ref{milddensegen}, there exists a dense subset
$\mathcal{H}'\subset\mathcal{H}_g(m_1,\ldots,m_\kappa)$ such that
for every $(M,\omega)\in\mathcal{H}'$ if $\mathcal{F}$ stands for
its vertical flow then the flows $\mathcal{F}^t$ and $\mathcal{F}$
are not isomorphic for every positive $t\neq 1$. Moreover, every
element of $\mathcal{H}'$ can be represented in the form
$M(\pi,\lambda,\tau)$. Fix $(M,\omega)\in\mathcal{H}'$ and real
$s\neq 0$. By $\tilde{\mathcal{F}}$ denote the vertical flow for
$\mathcal{T}_s(M,\omega)$. Let
$(\pi,\lambda,\tau)\in\hat{\mathcal{H}}(C)$ be a triple such that
$M(\pi,\lambda,\tau)=(M,\omega)$. Then
\[\mathcal{T}_s(M,\omega)=\mathcal{T}_sM(\pi,\lambda,\tau)=M(\mathcal{T}_s(\pi,\lambda,\tau))=M(\pi,e^s\lambda,e^{-s}\tau).\]
It follows that $\tilde{\mathcal{F}}$ is isomorphic to the special
flow
$T_{(\pi,e^s\lambda)}^{f_{h(e^{-s}\tau)}}=T_{(\pi,e^s\lambda)}^{e^{-s}f_{h(\tau)}}$.
Moreover, $T_{(\pi,e^s\lambda)}^{e^{-s}f_{h(\tau)}}$ is isomorphic
to $\left(T_{(\pi,\lambda)}^{f_{h(\tau)}}\right)^{e^s}$ via the
map $(x,y)\mapsto(e^{-s}x,e^sy)$. It follows that
$\tilde{\mathcal{F}}$ is isomorphic to ${\mathcal{F}}^{e^s}$.
Therefore $\tilde{\mathcal{F}}$ is not isomorphic to
$\mathcal{F}$.
\end{proof}
\section*{Acknowledgements} The author would like to thank M. Lema\'nczyk for inspiration,
E. Gutkin and the referee for their comments which made the text
more readable.

\end{document}